\documentclass[12pt]{amsart}

\usepackage{amssymb,stmaryrd}
\usepackage{booktabs}
\usepackage{cases}
\usepackage{array} 
\usepackage{float} 
\usepackage{rotating}

\usepackage{tikz}
\usetikzlibrary{positioning}
\usepackage{rank-2-roots}

\usepackage{color}
\usepackage{etoolbox} 
\usepackage{expl3}
\usepackage{pgfkeys}
\usepackage{pgfopts}
\usepackage{xparse}
\usepackage{xstring}

\usepackage[foot]{amsaddr}
\makeatletter
\renewcommand{\email}[2][]{%
  \ifx\emails\@empty\relax\else{\g@addto@macro\emails{,\space}}\fi%
  \@ifnotempty{#1}{\g@addto@macro\emails{\textrm{(#1)}\space}}%
  \g@addto@macro\emails{#2}%
}
\makeatother
\usepackage{mathtools}
\usepackage{enumerate,pdflscape}
\usepackage{multirow}
\usepackage{color}
  \usepackage[titletoc]{appendix}
\newcommand{\stt}{\mathfrak{s}_{T, \mathfrak{t}}}

\newcommand{\radfac}{\mathfrak{s}_{T^u, \mathfrak{k}^\perp}}

\makeatletter
\newcommand{\svast}{\bBigg@{3}}
\newcommand{\vast}{\bBigg@{4}}
\newcommand{\Vast}{\bBigg@{5}}
\makeatother

\newcommand{\cw}{[T \cup -T ]}
\newcommand{\cww}{T \cup -T }

  \theoremstyle{definition}
  \newtheorem{definition}{Definition}[section]
  \newtheorem{example}[definition]{Example}
  \AtEndEnvironment{example}{\null\hfill\qedsymbol}
  
\newcommand{\wpl}{\mathfrak{s} \inplus \mathfrak{r}}

  \newtheorem{remark}[definition]{Remark}
    \AtEndEnvironment{remark}{\null\hfill\qedsymbol}
  \theoremstyle{plain}
  \newtheorem{lemma}[definition]{Lemma}
  \newtheorem{proposition}[definition]{Proposition}
  \newtheorem{theorem}[definition]{Theorem}
  \newtheorem{corollary}[definition]{Corollary}
   \usepackage{setspace}
    
    \theoremstyle{definition}

\title[Cyclic wide subalgebras]{Cyclic wide subalgebras of semisimple Lie algebras}
  
\begin{document}

\author[Andrew Douglas]{Andrew Douglas$^{1,2}$}
\address[]{$^1$Department of Mathematics, New York City College of Technology, City University of New York, Brooklyn, NY, USA.}
\address[]{$^2$Ph.D. Program in Mathematics, CUNY Graduate Center, City University of New York, New York, NY, USA.}

\author[Joe Repka]{Joe Repka$^3$}
\address{$^3$Department of Mathematics, University of Toronto, Toronto, ON,  Canada.}

\date{\today}

\keywords{Levi decomposable subalgebras, cyclic indecomposable modules, wide subalgebras,  regular subalgebras, closed subsets of root systems.} 
\subjclass[2010]{17B05, 17B10, 17B20, 17B22,  17B30}

\begin{abstract}
Let $\wpl$ be a Levi decomposable Lie algebra, with Levi factor $\mathfrak{s}$, and radical $\mathfrak{r}$. A module $V$ of $\wpl$ is 
{\it cyclic indecomposable} if it is indecomposable and the quotient module $V /\mathfrak{r}\cdot V$ is a simple $\mathfrak{s}$-module. 
A Levi decomposable subalgebra of a semisimple Lie algebra is {\it cyclic wide} if the restriction of every simple module of the semisimple Lie algebra to the subalgebra
is cyclic indecomposable. We establish  a condition for a regular Levi decomposable subalgebra of a semisimple Lie algebra to be cyclic wide. 
Then,  in the case of a regular Levi decomposable subalgebra  
whose radical is an ad-nilpotent subalgebra, we show that the  condition is necessary and sufficient for the subalgebra to be cyclic wide.
All Lie algebras, and modules in this article are finite-dimensional, and over the complex numbers.
\end{abstract}

\maketitle

\section{Introduction}

Let $\wpl$ be a Levi decomposable Lie algebra, with Levi factor $\mathfrak{s}$, and radical $\mathfrak{r}$. A module $V$ of $\wpl$ is 
{\it cyclic indecomposable} if it is indecomposable and the quotient module $V /\mathfrak{r}\cdot V$ is a simple $\mathfrak{s}$-module.  Cyclic indecomposable modules were first considered by Piard  in $1986$ \cite{piard} as a means to examine  indecomposable modules of the Levi decomposable algebra $\mathfrak{sl}_2\inplus \mathbb{C}^2$.

A Levi decomposable subalgebra of a semisimple Lie algebra is {\it cyclic wide} if the restriction of every simple module of the semisimple Lie algebra to the subalgebra is cyclic indecomposable.
In \cite{casati}, Casati classified  the cyclic indecomposable modules of $\mathfrak{sl}_{n} \inplus \mathbb{C}^n$. Fundamental to the classification was
the proof that  all subalgebras of $\mathfrak{sl}_{n+1}$ isomorphic to $\mathfrak{sl}_{n} \inplus \mathbb{C}^n$ are cyclic wide,
although the term ``cyclic wide" was not used. 
  Note that all 
such subalgebras were shown to be regular  in \cite{dr15}.
The proof in \cite{casati}, although  elegant,  was also very technical,  employed  the Feigin-Fourier-Littelmann basis for $\mathfrak{sl}_{n+1}$-modules \cite{feigin}, and was tailored to the special case of $\mathfrak{sl}_{n+1}$.  
Note that  \cite{casati} utilizes results of \cite{dr15}.

In \cite{dr15},  Douglas and Repka classified, up to inner automorphism,  Levi decomposable  subalgebras inside simple Lie algebras. 
In particular, they classified subalgebras  isomorphic to subalgebras listed below  in the corresponding ambient simple Lie algebras:
\begin{equation}\label{ghwerr}
\begin{array}{llllllllllll}
\mathfrak{sl}_{n} &\inplus& \mathbb{C}^n&\subset& \mathfrak{sl}_{n+1}, \\
\mathfrak{so}_{2n+1} &\inplus& \mathbb{C}^{2n+1} &\subset& \mathfrak{so}_{2n+3}, \\
\mathfrak{sp}_{2n} &\inplus& \mathbb{C}^{2n}&\subset& \mathfrak{sp}_{2n+2}, ~\text{and}\\
\mathfrak{so}_{2n} &\inplus& \mathbb{C}^{2n}  &\subset& \mathfrak{so}_{2n+2}.
\end{array}
\end{equation}
Further, each subalgebra was shown to be regular.
It was shown, however, that there were no subalgebras isomorphic to $\mathfrak{sp}_{2n} \inplus \mathbb{C}^{2n}$ in $\mathfrak{sp}_{2n+2}$.

In this article, we generalize and extend the aforementioned result of  \cite{casati} regarding cyclic wide subalgebras. In particular, 
we establish a  condition for a regular Levi decomposable subalgebra of a semisimple Lie algebra to be cyclic wide. 
Then, we show this  condition is necessary and sufficient to determine whether or not  regular Levi decomposable subalgebras, 
whose radicals are ad-nilpotent subalgebras, are  cyclic wide.

Our results imply, for instance, that any  Levi decomposable subalgebra isomorphic to a subalgebra  of  the  corresponding ambient simple Lie algebra  in Eq. \eqref{ghwerr} is cyclic wide, including subalgebras of $\mathfrak{sl}_{n+1}$
isomorphic to $\mathfrak{sl}_{n} \inplus \mathbb{C}^n$ considered in  \cite{casati}. The proof of the result makes fundamental use of results in \cite{dr23, panyu}, which will be stated below.

The article is organized as follows. In Section \ref{background}, we review relevant background, and establish terminology, and notation.  Section \ref{cyclic} 
contains the results of the article.

Note that all Lie algebras, and modules in this article are finite-dimensional, and over the complex numbers.

\section{Background, terminology, and notation}\label{background}

In this section, we introduce terminology and notation, and review relevant background on semisimple Lie algebras and their regular subalgebras. 
 We  largely follow the background section of \cite{dr23}.

\subsection{Semisimple Lie algebras}

Let  $\mathfrak{g}$  denote a semisimple Lie algebra, and $\mathfrak{h}$ a fixed Cartan subalgebra of $\mathfrak{g}$. The corresponding root system is denoted $\Phi$, and its Weyl group $\mathcal{W}$. For $\alpha \in \Phi$,  we denote $\mathfrak{g}_\alpha$  the corresponding root space. Let $G$ be the adjoint group of $\mathfrak{g}$. 

Given $\alpha \in \Phi$, we may  fix a nonzero $e_\alpha \in \mathfrak{g}_\alpha$. Then  there 
 is a unique $e_{-\alpha} \in \mathfrak{g}_{-\alpha}$, such that $e_\alpha$, $e_{-\alpha}$, and $h_{\alpha} =[e_\alpha, e_{-\alpha}] \in \mathfrak{h}$
 satisfy the commutation relations $[h_\alpha, e_\alpha]=\alpha(h_\alpha) e_\alpha=2 e_\alpha$, and $[h_\alpha, e_{-\alpha}]=-2 e_{-\alpha}$. 
 Therefore, 
 $e_\alpha$, $e_{-\alpha}$, and $h_\alpha$ form a basis for a subalgebra 
isomorphic to $\mathfrak{sl}_2$.

The positive roots of $\Phi$ are denoted  $\Phi^+$, and $\Delta =\{ \alpha_1,...,\alpha_n\}\subseteq\Phi^+$ is a base of $\Phi$. The elements of $\Delta$ are 
{\it simple roots}.
We define $\Phi^- \coloneqq -\Phi^+$. The {\it rank} of $\mathfrak{g}$  is the cardinality of $\Delta$, which in this case is $n$.

The Weyl group $\mathcal{W}$ is generated by the  reflections $ s_{\alpha_1},..., s_{\alpha_n}$, corresponding to the simple roots of $\Delta$.
If an element $w\in \mathcal{W}$ is written as $s_{\beta_1}\cdots s_{\beta_t}$, such that $\beta_i \in \Delta$  and $t$  is minimal, the expression is {\it reduced}.
Then, we define the {\it length} of $w$, relative to $\Delta$, as $l(w)=t$. There is a unique longest element of $\mathcal{W}$, denoted $w_0$.  Further, 
$w_0$ sends $\Phi^+$ to $\Phi^-$ (see \cite{stek} or \cite{humphreys}).

We denote by $\Lambda$ the set of weights relative to the root system $\Phi$. Further, $\Lambda^+$ is the set of dominant weights with respect to $\Delta$.
Let $\lambda_1,..., \lambda_n$ be the {\it fundamental dominant weights} (relative to $\Delta$).

A dominant  weight $\lambda \in \Lambda^+$ may be written as $\lambda = m_1 \lambda_1+ \cdots +m_n \lambda_n$, where $m_i$ is a nonnegative integer for each $i$. For each dominant weight $\lambda$, $V(\lambda)$ is  the simple $\mathfrak{g}$-module of highest weight $\lambda$. 

For an arbitrary $\mathfrak{g}$-module $V$,  let  $\Pi(V)$ be the set of weights of $V$. Then,
 $V$
decomposes into weight spaces
\begin{equation}
 V = \bigoplus_{\mu \in \Pi(V) } V_\mu,
\end{equation}
where $V_\mu =\{ v\in V ~|~ h \cdot v = \mu(h) v, ~ \text{for all}~ h\in \mathfrak{h} \}$. Fix a highest weight vector $v_\lambda \in V(\lambda)$, unique up to scalar multiple.

The Cartan subalgebra $\mathfrak{h}$ is naturally associated with its dual space $\mathfrak{h}^*$ via the Killing form $\kappa$. In particular,
$\alpha \in \mathfrak{h}^*$ 
corresponds to the unique element $t_\alpha \in \mathfrak{h}$ such that $\alpha(h) = \kappa(t_\alpha, h)$, for all $h \in \mathfrak{h}$. A  nondegenerate symmetric bilinear form  on $\mathfrak{h}^*$ may be defined by
$(\alpha, \beta) \coloneqq \kappa(t_{\alpha}, t_{\beta} )$. Further,  we define  $\langle \alpha, \beta \rangle \coloneqq \frac{2(\alpha, \beta)}{(\beta, \beta)} =\alpha (h_{\beta})$, where 
$h_{\beta} \coloneqq \frac{2t_\beta}{\kappa(t_\beta, t_\beta)}=\frac{2t_\beta}{(\beta, \beta)}$.

\subsection{Closed subsets of root systems and regular subalgebras}

Closed subsets of root systems are fundamentally related to regular subalgebras.  
A subset $T$ of the root system $\Phi$ is {\it closed} if for any
$x, y \in T$, $x+y \in \Phi$ implies $x+y \in T$. 

A closed subset $T$ is a disjoint union of its {\it symmetric} component 
\begin{equation}
T^r =\{\alpha \in T | -\alpha \in T  \},
\end{equation} 
and its {\it special} component 
\begin{equation}
T^u=\{ \alpha \in T | -\alpha  \notin T \}.
\end{equation}
We have the following useful results on closed subsets of root systems.
\begin{lemma}\label{lem:deccl}\cite{sopkina}
$T^r$ is a closed root subsystem  of $\Phi$. For any two roots $\alpha \in T^u$ and $\beta \in T$ such that $\alpha+\beta$ is a root, we have that  $\alpha+\beta \in T^u$. In particular, $T^u$ is closed.
\end{lemma}
\begin{lemma}\label{wideclosure}\cite{dr23}
Let $S$ be a closed subset of $\Phi$. Suppose $\beta_1$, $\beta_2$,...,$\beta_k \in S$, and $\beta_1+\beta_2+ \cdots+\beta_k \in \Phi$. Then $\beta_1+\beta_2+\cdots +\beta_k \in S$. 
\end{lemma}

With the following definition, we have another relevant result on closed subsets of root systems. Let $S \subseteq \Phi$, then the {\it closure} of $S$, denoted $[S]$, is the smallest closed subset of $\Phi$ containing $S$.  
\begin{lemma}\label{semireg}\cite{dr23}
Let $T$ be a closed subset of $\Phi$. Then, $[T\cup -T]$ is a symmetric closed subset of $\Phi$ containing $T$. 
\end{lemma}

Closed subsets $T$ and $T'$ of a root system $\Phi$ are {\it conjugate} if there exists an element $w\in \mathcal{W}$ such that
$w(T)=T'$.

\begin{lemma}\label{lem:1b}\cite{sopkina, bou} 
A special closed subset is conjugate to a subset of $\Phi^+$.
\end{lemma}

Let $T\subseteq \Phi$ be a closed subset, and $\mathfrak{t}$ be a subspace of $\mathfrak{h}$
containing $[\mathfrak{g}_\alpha, \mathfrak{g}_{-\alpha}]$ for each $\alpha \in T^r$. Then
\begin{equation}\label{reggg}
\mathfrak{s}_{T,\mathfrak{t}} = \mathfrak{t} \oplus \bigoplus_{\alpha \in T} \mathfrak{g}_\alpha  
\end{equation}
is a regular subalgebra of $\mathfrak{g}$.  In addition, a regular subalgebra of $\mathfrak{g}$ normalized by $\mathfrak{h}$ may be written in this form.
Further, any regular subalgebra of $\mathfrak{g}$ is conjugate under $G$ to a regular subalgebra normalized by $\mathfrak{h}$. 
This justifies our assumption henceforth that any regular subalgebra is normalized by 
$\mathfrak{h}$, and thus may be written in the form of Eq. \eqref{reggg}.

The Weyl group $\mathcal{W}$ of $\Phi$ naturally acts on the dual space $\mathfrak{h}^*$. Further, by identifying  $\mathfrak{h}$ and $\mathfrak{h}^*$ via the Killing form, the Weyl group also acts on $\mathfrak{h}$.

\begin{proposition}\label{wprop}[ \cite{dougdeg}, Proposition 5.1]
The regular subalgebras $\mathfrak{s}_{T_1, \mathfrak{t}_1}$ and  $\mathfrak{s}_{T_2, \mathfrak{t}_2}$ are conjugate under $G$ if and only if there is a $w \in \mathcal{W}$ with 
$w(T_1)=w(T_2)$ and $w( \mathfrak{t}_1)=w( \mathfrak{t}_2)$.
\end{proposition}

By Levi's theorem, a Lie algebra must be either semisimple, solvable, or Levi decomposable [\cite{levi}, Chapter II, Section $2$]. 
Given a non-empty, symmetric closed subset $T$, and a subalgebra $\mathfrak{t}$ of $\mathfrak{h}$ generated by $[\mathfrak{g}_\alpha, \mathfrak{g}_{-\alpha}]$ for all $\alpha \in T\cap \Phi^+$, then $\mathfrak{s}_{T, \mathfrak{t}}$ is a regular semisimple Lie algebra. 
If $T$ is a special closed subset of $\Phi$ (possibly empty), and $\mathfrak{t}$ is a subalgebra of $\mathfrak{h}$, then $\mathfrak{s}_{T,\mathfrak{t}}$ is 
regular solvable subalgebra.

Regular Levi decomposable subalgebras are the focus of the current article.
 A regular Levi decomposable subalgebra is a non-semisimple regular subalgebra $\mathfrak{s}_{T, \mathfrak{t}}$, 
such that $T$ is a closed subset, $T^r$ is the corresponding non-empty, symmetric closed subset, and $T^u$ is the corresponding  special closed subset of $\Phi$, and  $\mathfrak{t}$ is  a subalgebra of $\mathfrak{h}$ containing $[\mathfrak{g}_\alpha, \mathfrak{g}_{-\alpha} ]$ for each $\alpha \in T^r$.

Considering Lemma \ref{lem:deccl}, a regular Levi decomposable subalgebra may be written as $\stt =\mathfrak{s}_{T^r, \mathfrak{k}} \inplus \mathfrak{s}_{T^u, \mathfrak{k}^\perp}$, where $\mathfrak{k}$ is generated by $[\mathfrak{g}_\alpha, \mathfrak{g}_{-\alpha}]$ for all $\alpha \in T^r$; and $\mathfrak{k}^\perp =\{ h\in \mathfrak{t} ~|~ \kappa(h,h') = 0,~\text{for all}~ h'\in \mathfrak{k} \}$. Since our definition of a regular Levi decomposable subalgebra $\stt$ 
requires a non-trivial Levi factor, and excludes the 
possibility that $\stt$ is semisimple, we must have
\begin{equation}\label{levipos}
T^u \neq \emptyset; ~\text{or} ~ T^u = \emptyset ~\text{and}~ \mathfrak{k}^\perp\neq \{0\}.
\end{equation}
Note that $\mathfrak{s}_{T^r, \mathfrak{k}}$ is the Levi factor of $\stt$, and $\mathfrak{s}_{T^u, \mathfrak{k}^\perp}$ is the radical of $\stt$.
As a special case, if $\mathfrak{k}^\perp=\{0\}$, then $\stt =\mathfrak{s}_{T^r, \mathfrak{t}} \inplus \mathfrak{s}_{T^u, 0}$.
Note also that regular Levi decomposable subalgebras $\stt=\mathfrak{s}_{T^r,\mathfrak{k}} \inplus \mathfrak{s}_{T^u,\mathfrak{k}^\perp}$ of $\mathfrak{g}$, whose radicals $\mathfrak{s}_{T^u,\mathfrak{k}^\perp}$ are  ad-nilpotent subalgebras of $\mathfrak{g}$, are precisely those with $\mathfrak{k}^\perp=\{0\}$.

A subalgebra  of a semisimple Lie algebra  is {\it wide} if every simple module of the semisimple Lie algebra remains indecomposable when restricted to the subalgebra. 
In particular, a cyclic wide subalgebra is wide. We record one additional crucial result establishing necessary and sufficient conditions for a regular subalgebra to be wide.

\begin{theorem}\label{lwideb}\cite{dr23, panyu}
Let $T$ be  a closed subset of $\Phi$, and $\mathfrak{t}$  a subalgebra of $\mathfrak{h}$ containing $[\mathfrak{g}_\alpha, \mathfrak{g}_{-\alpha} ]$ for each $\alpha \in T^r$.  Then, $\mathfrak{s}_{T, \mathfrak{t}}$ is wide if and only if $[T \cup -T] =\Phi$. 
\end{theorem}

Note that necessary and sufficient conditions for essentially all regular solvable subalgebras to be wide were established in \cite{panyu}. The result was extended
to all regular subalgebras in \cite{dr23}, including the regular Levi decomposable subalgebras, which are the focus of the present article.

\section{Cyclic wide subalgebras} \label{cyclic}

In this section, we establish  a condition for a regular 
  Levi decomposable subalgebra of a semisimple Lie algebra 
 $\mathfrak{g}$ 
 to be  cyclic wide (Theorem \ref{theorem1}).  
Then, we 
consider regular Levi decomposable subalgebras  
whose radicals are ad-nilpotent subalgebras, and show that the 
condition is necessary and sufficient to determine whether or not such subalgebras  
are  cyclic wide (Theorem \ref{theoremtwo}).
 Several examples are presented to
 illustrate the results.  
We begin with three lemmas.

\begin{lemma}\label{lemma3}
A  regular Levi decomposable subalgebra $\mathfrak{s}_{\widetilde{T}, \mathfrak{\widetilde{t}}}$ of $\mathfrak{g}$ is conjugate under the adjoint group $G$ of $\mathfrak{g}$ to 
a  Levi decomposable subalgebra $\stt= \mathfrak{s}_{T^r,\mathfrak{k}} \inplus \mathfrak{s}_{T^u,\mathfrak{k}^\perp}$, such that 
\begin{equation}
T^u\neq \emptyset~\text{and}~  T^u \subseteq \Phi^-; ~\text{or}~ T^u=\emptyset~\text{and}~ \mathfrak{k}^\perp\neq \{0\}.
\end{equation} 
\end{lemma}
\begin{proof}
Let  $\mathfrak{s}_{\widetilde{T}, \mathfrak{\widetilde{t}}}=\mathfrak{s}_{\widetilde{T}^r, \mathfrak{\widetilde{k}}} \inplus \mathfrak{s}_{\widetilde{T}^u, \mathfrak{\widetilde{k}}^\perp}$ be a 
regular Levi decomposable subalgebra, for the subalgebra $\mathfrak{\widetilde{k}}$ of $\mathfrak{\widetilde{t}}$ 
 generated by $[\mathfrak{g}_\alpha, \mathfrak{g}_{-\alpha}]$, for all $\alpha  \in \widetilde{T}^r$. We assume that it's not the case that $\widetilde{T}^u=\emptyset$ and $\mathfrak{\widetilde{k}}^\perp\neq \{0\}$, for otherwise there is nothing to prove. Thus, we have $\widetilde{T}^u \neq \emptyset$, considering Eq. \eqref{levipos}.

By Lemma \ref{lem:1b}, there exists $w'$ in the Weyl group  $\mathcal{W}$ such that $w'(\widetilde{T}^u) \subseteq \Phi^+$. Then,  applying
the longest element $w_0$ of $\mathcal{W}$ (relative to $\Delta$), we have $w_0 \circ w'(\widetilde{T}^u) \subseteq \Phi^-$ (see \cite{stek} or \cite{humphreys}).  
Hence, 
\begin{equation}
\mathfrak{s}_{w_0\circ w'(\widetilde{T}),  w_0 \circ w'(\mathfrak{\widetilde{t}})}   =   \mathfrak{s}_{w_0 \circ w'(\widetilde{T}^r),w_0 \circ w'(\mathfrak{\widetilde{k}})} \inplus \mathfrak{s}_{w_0 \circ w'(\widetilde{T}^u),w_0 \circ w'(\mathfrak{\widetilde{k}}^\perp)}
\end{equation}
has the desired property  that the radical is determined by a closed subset $w_0 \circ w'(\widetilde{T}^u)$ contained in $\Phi^-$.
And, by Proposition \ref{wprop}, it is conjugate under $G$ to $\mathfrak{s}_{ \widetilde{T}, \mathfrak{\widetilde{t}}}$.
\end{proof}

\begin{lemma}\label{lemma22}
A  regular Levi decomposable subalgebra   $\mathfrak{s}_{\widetilde{T}, \mathfrak{\widetilde{t}}}$, such that  $\widetilde{T}\cup -\widetilde{T}=\Phi$, is 
conjugate under the adjoint group $G$ of $\mathfrak{g}$ to 
a  Levi decomposable subalgebra $\stt= \mathfrak{s}_{T^r,\mathfrak{k}} \inplus \mathfrak{s}_{T^u,\mathfrak{k}^\perp}$, such that 
\begin{equation}
T^u\neq \emptyset~\text{and}~  T^u \subseteq \Phi^-.
\end{equation} 
\end{lemma}
\begin{proof}
By Lemma \ref{lemma3}, $\mathfrak{s}_{\widetilde{T}, \mathfrak{\widetilde{t}}}$
is conjugate under the adjoint group $G$ of $\mathfrak{g}$ to 
a  subalgebra $\stt= \mathfrak{s}_{T^r,\mathfrak{k}} \inplus \mathfrak{s}_{T^u,\mathfrak{k}^\perp}$, such that 
\begin{equation}
T^u\neq \emptyset~\text{and}~  T^u \subseteq \Phi^-; ~\text{or}~ T^u=\emptyset~\text{and}~ \mathfrak{k}^\perp\neq \{0\}.
\end{equation} 
Since $\widetilde{T}\cup -\widetilde{T} =\Phi$, then $T\cup -T =\Phi$.   Since $T\cup -T =\Phi$, it follows, as we'll show, that we must have $T^u\neq \emptyset$ and  $T^u \subseteq \Phi^-$: This follows since,
if $T^u=\emptyset$,   $\mathfrak{k}^\perp\neq \{0\}$, and  $T\cup -T =\Phi$, then $T\cup -T=T =\Phi$. But, then we'd have $\stt=\mathfrak{g}$, a contradiction.
Thus, it must be the case that  $T^u\neq \emptyset$ and  $T^u \subseteq \Phi^-$.
\end{proof}

 Note that a closed subset $T$ of $\Phi$ is  called {\it parabolic} if $T \cup -T =\Phi$. However, we didn't adopt this terminology.

\begin{remark}\label{remark}
Justified by Lemma \ref{lemma3}, we will assume henceforth that a regular Levi decomposable subalgebra 
$\stt= \mathfrak{s}_{T^r,\mathfrak{k}} \inplus \mathfrak{s}_{T^u,\mathfrak{k}^\perp}$ is such that
\begin{equation}
T^u\neq \emptyset ~\text{and}~T^u \subseteq \Phi^-;~\text{or} ~ T^u=\emptyset~ \text{and}~\mathfrak{k}^\perp\neq \{0\}.
\end{equation}
And, justified by Lemma \ref{lemma22}, we will assume henceforth that for a regular Levi decomposable subalgebra $\stt$ such that $\cww =\Phi$, then
\begin{equation}
T^u\neq \emptyset ~\text{and}~T^u \subseteq \Phi^-.
\end{equation}
\end{remark}

Let $\stt= \mathfrak{s}_{T^r,\mathfrak{k}} \inplus \mathfrak{s}_{T^u,\mathfrak{k}^\perp}$ be a regular Levi decomposable subalgebra of $\mathfrak{g}$. 
Given the  simple $\mathfrak{g}$-module $V(\lambda)$, with fixed highest weight vector $v_\lambda\in V(\lambda)$, 
we identify two $\mathfrak{s}_{T^r,\mathfrak{k}}$-submodules of $V(\lambda)$. The first is 
$\mathfrak{s}_{T^u, \mathfrak{k}^\perp} \cdot V(\lambda)$, and the second is
\begin{equation}
\langle v_\lambda \rangle_{\mathfrak{s}_{T^r,\mathfrak{k}}}\coloneqq \text{Span} \{e_{-\beta_1} \cdots e_{-\beta_m} \cdot v_\lambda~|~ -\beta_1,...,-\beta_m \in T^r \cap \Phi^- \},
\end{equation}
where $\beta_1$,...,$\beta_m$ are not necessarily distinct. Note that we include $v_\lambda \in \langle v_\lambda \rangle_{\mathfrak{s}_{T^r,\mathfrak{t}}}$.
Observe that the $\mathfrak{s}_{T^r,\mathfrak{k}}$-submodule $\langle v_\lambda \rangle_{\mathfrak{s}_{T^r,\mathfrak{k}}}$ is simple
since it is generated by a highest weight vector $v_\lambda$, with respect to the semisimple Lie algebra $\mathfrak{s}_{T^r,\mathfrak{k}}$ 
(see [\cite{humphreys}, Theorem 20.2]).

\begin{lemma}\label{lemma}
Let  $\stt = \mathfrak{s}_{T^r,\mathfrak{k}} \inplus \mathfrak{s}_{T^u,\mathfrak{k}^\perp}$ be a regular Levi decomposable subalgebra of $\mathfrak{g}$, such that
$\cww=\Phi$.
 Further, let $V(\lambda)$ be the simple $\mathfrak{g}$-module of highest weight $\lambda$. Then,  
 we have a decomposition of $\mathfrak{s}_{T^r,\mathfrak{k}}$-modules
\begin{equation}
V(\lambda)|_{\mathfrak{s}_{T^r,\mathfrak{k}} } = (\radfac \cdot V(\lambda)) \oplus \langle v_\lambda \rangle_{\mathfrak{s}_{T^r,\mathfrak{k}}},
~\text{or} ~  \radfac \cdot V(\lambda)=V(\lambda).
\end{equation}
\end{lemma}
\begin{proof}
Let $\stt$ be a regular Levi decomposable subalgebra such that $\cww=\Phi$. Then, we have $T^u\neq \emptyset$ and  $T^u \subseteq \Phi^-$ (see Remark \ref{remark}).
This implies that $\Phi^-\subseteq T$.

We first show that as vector spaces $V(\lambda)= (\radfac \cdot V(\lambda)) + \langle v_\lambda \rangle_{\mathfrak{s}_{T^r,\mathfrak{k}}}$. Consider a monomial weight vector $e_{-\beta_1} \cdots e_{-\beta_l}  v_\lambda \neq 0$ of $V(\lambda)$ with $-\beta_1,...,-\beta_l \in \Phi^- \subseteq T$ not necessarily distinct.
Such elements generate $V(\lambda)$ [\cite{humphreys}, Theorem 20.2]. Thus, it suffices to show $e_{-\beta_1} \cdots e_{-\beta_l}  v_\lambda \in (\radfac \cdot V(\lambda)) + \langle v_\lambda \rangle_{\mathfrak{s}_{T^r,\mathfrak{k}}}$.

If $-\beta_1,...,-\beta_l \in T^r \cap \Phi^-$, then $e_{-\beta_1} \cdots e_{-\beta_l}  v_\lambda \in \langle v_\lambda \rangle_{\mathfrak{s}_{T^r,\mathfrak{k}}}$. 
If $-\beta_1 \in T^u\subseteq \Phi^-$, then $e_{-\beta_1} \cdots e_{-\beta_l}  v_\lambda \in \radfac \cdot V(\lambda)$. 

Otherwise, there is a minimum $m$  with $1\leq m \leq l-1$ such that 
$-\beta_1,...,-\beta_m  \in T^r \cap \Phi^-$, and $-\beta_{m+1} \in T^u \subseteq \Phi^-$.  We will show that $e_{-\beta_1} \cdots e_{-\beta_l} \cdot v_\lambda \in \radfac \cdot V(\lambda)$. 
We have
\begin{equation} \label{summm}
\setstretch{1.35}
\begin{array}{lllll}
e_{-\beta_1} \cdots e_{-\beta_m}  e_{-\beta_{m+1}} \cdots e_{-\beta_l}  v_\lambda =\\
 e_{-\beta_1} \cdots e_{-\beta_{m-1}} e_{-\beta_{m+1}} e_{-\beta_m} \cdots e_{-\beta_l}  v_\lambda ~~ + \\
  e_{-\beta_1} \cdots e_{-\beta_{m-1}} [ e_{-\beta_m}, e_{-\beta_{m+1}} ] \cdots e_{-\beta_l}  v_\lambda. 
\end{array}
\end{equation}
If $[ e_{-\beta_m}, e_{-\beta_{m+1}} ] \neq0$, then $[ e_{-\beta_m}, e_{-\beta_{m+1}} ] \in \mathfrak{g}_{-\gamma}$, where $-\gamma \in T^u \subseteq \Phi^-$, considering Lemma \ref{lem:deccl}.
   If we continue this procedure on  the two 
 terms on the right-hand side of 
  Eq. \eqref{summm}, we end with a sum of terms, each of whose first factor is a negative root vector with root in $T^u \subseteq \Phi^-$. Hence, $e_{-\beta_1} \cdots e_{-\beta_l} v_\lambda$ is an element of
  $\radfac \cdot V(\lambda)$. 
Therefore, we've established the equality of vector spaces $V(\lambda)= (\radfac \cdot V(\lambda)) + \langle v_\lambda \rangle_{\mathfrak{s}_{T^r,\mathfrak{k}}}$.

Now, if $(\radfac \cdot V(\lambda)) \cap \langle v_\lambda \rangle_{\mathfrak{s}_{T^r,\mathfrak{k}}} =\{0\}$, then $V(\lambda)= (\radfac \cdot V(\lambda)) \oplus \langle v_\lambda \rangle_{\mathfrak{s}_{T^r,\mathfrak{k}}}$. If 
$(\radfac \cdot V(\lambda)) \cap \langle v_\lambda \rangle_{\mathfrak{s}_{T^r,\mathfrak{k}}} \neq \{0\}$, then, 
since $\langle v_\lambda \rangle_{\mathfrak{s}_{T^r,\mathfrak{k}}}$ is a simple $\mathfrak{s}_{T^r,\mathfrak{k}}$-module, we have $\langle v_\lambda \rangle_{\mathfrak{s}_{T^r,\mathfrak{k}}}\subseteq  \radfac \cdot V(\lambda)$. In this case, considering $V(\lambda)= (\radfac \cdot V(\lambda)) + \langle v_\lambda \rangle_{\mathfrak{s}_{T^r,\mathfrak{k}}}$, we have
$V(\lambda)=\radfac \cdot V(\lambda)$. 
\end{proof}

We will now present the theorem that establishes conditions for a regular Levi decomposable subalgebra to be cyclic wide.

\begin{theorem}\label{theorem1}
Let $\stt$ be a  regular Levi decomposable subalgebra
 of $\mathfrak{g}$. If $\cww=\Phi$, then $\stt$ is cyclic wide. 
\end{theorem}
\begin{proof} Let $\stt= \mathfrak{s}_{T^r,\mathfrak{k}} \inplus \mathfrak{s}_{T^u,\mathfrak{k}^\perp}$ be a regular Levi decomposable subalgebra
such that $\cww=\Phi$. Then,  we have $T^u \neq \emptyset$ and $T^u \subseteq \Phi^-$ (see Remark \ref{remark}).
Further, since $\cww =\Phi$, then $\cw =\Phi$, so that $\stt$ is wide by Theorem \ref{lwideb}.

 Now,  let $V(\lambda)$ be the simple $\mathfrak{g}$-module of highest weight $\lambda$. Since $\stt$ is wide, $V(\lambda)$ is indecomposable when restricted to $\stt$. We must show that 
 $ V(\lambda) / ( \mathfrak{s}_{T^u,\mathfrak{k}^\perp}\cdot V(\lambda)) $ is a simple $\mathfrak{s}_{T^r, \mathfrak{k}}$-module.
 
 By Lemma \ref{lemma}, we have an isomorphism of $\mathfrak{s}_{T^r, \mathfrak{k}}$-modules: 
 \begin{equation}
 V(\lambda) / ( \mathfrak{s}_{T^u,\mathfrak{k}^\perp}\cdot V(\lambda)) \cong \langle v_\lambda \rangle_{\mathfrak{s}_{T^r,\mathfrak{k}}}, ~\text{or}~ V(\lambda) / ( \mathfrak{s}_{T^u,\mathfrak{k}^\perp}\cdot V(\lambda)) \cong V(0).
 \end{equation}
 Since both  $\langle v_\lambda \rangle_{\mathfrak{s}_{T^r,\mathfrak{k}}}$ and   $V(0)$ are simple $\mathfrak{s}_{T^r, \mathfrak{k}}$-modules, then  
  we must have that $ V(\lambda) / ( \mathfrak{s}_{T^u,\mathfrak{k}^\perp}\cdot V(\lambda)) $ is a simple $\mathfrak{s}_{T^r, \mathfrak{k}}$-module, as required.
\end{proof}

The following example illustrates that a regular Levi decomposable subalgebra $\stt$ that is wide  is not necessarily cyclic wide.
Hence, we could not weaken the hypothesis of Theorem \ref{theorem1} from $\cww=\Phi$ to $\cw=\Phi$.

\begin{example}
The special linear algebra $\mathfrak{sl}_4$ has simple roots $\Delta=\{ \alpha_1, \alpha_2, \alpha_3\}$ and positive roots
\begin{equation}
\Phi^+= \{ \alpha_{p, q} \coloneqq \alpha_p+\cdots + \alpha_q~|~ 1 \leq p \leq q \leq 3\}.
\end{equation}
Note that $\alpha_i =\alpha_{i,i}$.  Define a closed subset $T$ of $\Phi$ by specifying its symmetric and special parts, respectively:
\begin{equation}
\begin{array}{llll}
 T^r\coloneqq \{ \pm \alpha_3 \}, ~
   T^u\coloneqq \{ -\alpha_1, -\alpha_{1,2}, -\alpha_{1,3}\}.
   \end{array}
\end{equation}
Let $\mathfrak{t}$ be generated by $[e_{\alpha_3}, e_{-\alpha_3}]$.
Then, $\cw=\Phi$, so that  $\stt$ is wide (Theorem \ref{lwideb}).  However,  $\cww \subsetneq \Phi$.
Observe that the Levi factor $\mathfrak{s}_{T^r, \mathfrak{t}}$ is isomorphic to $\mathfrak{sl}_2$. And,
 the radical  $\mathfrak{s}_{T^u,0}$ is abelian.

Consider the simple $\mathfrak{sl}_4$-module $V(\lambda_3)$. Since  $\stt$ is wide,
$V(\lambda_3)$ is $\stt$-indecomposable. However, as we'll see, $V(\lambda_3)/ \mathfrak{s}_{T^u,0} \cdot V(\lambda_3)$ is not 
$\mathfrak{s}_{T^r,\mathfrak{t}}$-simple.

A basis for $V(\lambda_3)$ is given by $\{ v_{\lambda_3}, e_{-\alpha_3} v_{\lambda_3}, e_{-\alpha_{2,3}} v_{\lambda_3}, e_{-\alpha_{1,3}} v_{\lambda_3} \}$.
A decomposition of $V(\lambda_3)$ with respect to $\mathfrak{s}_{T^r, \mathfrak{t}}$ into $\mathfrak{s}_{T^r, \mathfrak{t}}$-submodules  is as follows:
\begin{equation}\label{decom1}
\begin{array}{llllllllll}
V(\lambda_3)|_{\mathfrak{s}_{T^r, \mathfrak{t}}} &=& \text{Span}\{  v_{\lambda_3}, e_{-\alpha_3} v_{\lambda_3} \}    \oplus    \text{Span}\{ e_{-\alpha_{2,3}} v_{\lambda_3} \}\oplus \\
&&  \text{Span}\{  e_{-\alpha_{1,3}} v_{\lambda_3} \}.
\end{array}
\end{equation}
And, 
\begin{equation}\label{decom2}
\mathfrak{s}_{T^u,0} \cdot V(\lambda_3) = \text{Span}\{ e_{-\alpha_{1,3}} v_{\lambda_3} \}.
\end{equation}
Eqs. \eqref{decom1} and \eqref{decom2} imply that  $V(\lambda_3)/ \mathfrak{s}_{T^u,0} \cdot V(\lambda_3)$ 
is isomorphic to $\text{Span}\{  v_{\lambda_3}, e_{-\alpha_3} v_{\lambda_3} \}    \oplus      \text{Span}\{  e_{-\alpha_{2,3}} v_{\lambda_3} \}$, as $\mathfrak{s}_{T^r, \mathfrak{t}}$-modules. 
Therefore, the $\mathfrak{s}_{T^r, \mathfrak{t}}$-module 
$V(\lambda_3)/ \mathfrak{s}_{T^u,0} \cdot V(\lambda_3)$  is not
simple.  This implies that, although $\stt$ is wide, it is not cyclic wide.

We make a final observation in this example. If we were to let $\mathfrak{t}=\text{Span} \{h_{\alpha_3}, 2h_{\alpha_2}+h_{\alpha_3} \}$, then  
$\stt=\mathfrak{s}_{T^r,\mathfrak{k}} \inplus \mathfrak{s}_{T^u,\mathfrak{k}^\perp}$, where
$\mathfrak{k}= \text{Span}\{h_{\alpha_3}\}$ and  $\mathfrak{k}^\perp= \text{Span}\{ 2h_{\alpha_2}+h_{\alpha_3}\}$. 
We would then have $\mathfrak{s}_{T^u,\mathfrak{k}^\perp} \cdot V(\lambda_3) =V(\lambda_3)$. 
Hence,  the $\mathfrak{s}_{T^r, \mathfrak{k}}$-module 
$V(\lambda_3)/ \mathfrak{s}_{T^u,\mathfrak{k}^\perp} \cdot V(\lambda_3) \cong V(0)$  would be
simple, and  thus $V(\lambda_3)$ would be a cyclic  indecomposable $\stt$-module. Note also that $\stt$ is not perfect.
\end{example}

In the next theorem, we prove that the condition $\cww=\Phi$ is 
 necessary and sufficient to determine whether or not   regular Levi decomposable subalgebras, 
whose radicals are ad-nilpotent subalgebras, are   cyclic wide.

\begin{theorem}\label{theoremtwo}
Let $\stt$ be a regular Levi decomposable subalgebra whose radical is an ad-nilpotent subalgebra. Then,  $\stt$ is cyclic wide if and only if $\cww=\Phi$.
\end{theorem}
\begin{proof}
$(\Longleftarrow)$ Since $\stt$ is a regular Levi decomposable subalgebra such that $\cww=\Phi$, then $\stt$ is cyclic wide by Theorem \ref{theorem1}. 

\vspace{2mm}

\noindent $(\Longrightarrow)$ Assume that $\stt$ is a cyclic wide, regular Levi decomposable subalgebra, whose radical is an ad-nilpotent subalgebra.  
Recall that regular Levi decomposable subalgebras $\stt=\mathfrak{s}_{T^r,\mathfrak{k}} \inplus \mathfrak{s}_{T^u,\mathfrak{k}^\perp}$ of $\mathfrak{g}$, whose radicals $\mathfrak{s}_{T^u,\mathfrak{k}^\perp}$ are  ad-nilpotent subalgebras of $\mathfrak{g}$, are precisely those with $\mathfrak{k}^\perp=\{0\}$. Hence, considering Remark \ref{remark},
we have $T^u \subseteq \Phi^-$.

Since $\stt$ is wide, then $\cw =\Phi$ by Theorem \ref{lwideb}. 
By way of contradiction, suppose that  $T\cup -T \subsetneq \Phi$. 

Then, for $\Delta=\{\alpha_1,...,\alpha_n \}$, there must exist $-\alpha_l\in -\Delta \setminus T$, as we'll now show. Suppose it's
not the case that there exists $-\alpha_l\in -\Delta \setminus T$. Then, $-\Delta \subseteq T$. It follows that $\Phi^-\subseteq T$, since $T$ is closed (see
[\cite{humphreys},Corollary $10.2$] applied to negative roots).  This, however, implies
that $T\cup -T= \Phi$, a contradiction. Thus, it must be the case that there exists $-\alpha_l\in -\Delta \setminus T$.
Observe also that
$\alpha_l \notin T$. 

Consider the simple $\mathfrak{g}$-module $V(\lambda_l)$.  We will  identify two  highest weight 
vectors of $V(\lambda_l)$ with respect to the semisimple regular subalgebra $\mathfrak{s}_{T^r,\mathfrak{k}}$, which are of different weights, and neither of these weights occur among weights of
$\mathfrak{s}_{T^u,0}\cdot V(\lambda_l)$.  In particular, this implies that these highest weight vectors are not
contained in $\mathfrak{s}_{T^u,0}\cdot V(\lambda_l)$. 

We have $e_{-\alpha_l} v_{\lambda_l} \neq 0$. Further,
$e_\beta e_{-\alpha_l} v_{\lambda_l} =0$, for all $\beta \in T^r\cap \Phi^+$, as we'll now see: 
We have $e_\beta e_{-\alpha_l} v_{\lambda_l}=[e_\beta, e_{-\alpha_l}] v_{\lambda_l}$, since $\beta \in \Phi^+$. However, 
 $\beta-\alpha_l$ cannot be
a negative root since $\beta\in \Phi^+$ and $-\alpha_l\in -\Delta$; and $\beta\neq \alpha_l$. Hence, $e_\beta e_{-\alpha_l} v_{\lambda_l} =0$, for all $\beta \in T^r\cap \Phi^+$.
Thus, $ e_{-\alpha_l} v_{\lambda_l}$ is a highest weight vector with respect to $\mathfrak{s}_{T^r,\mathfrak{k}}$ of weight $\lambda_l-\alpha_l$. 

Next we'll show that the weight of 
 $e_{-\alpha_l} v_{\lambda_l}$ does not occur among weights of $\mathfrak{s}_{T^u,0}\cdot V(\lambda_l)$. 
 If this were not the case,
 then
 $-\alpha_l =-\gamma-\mu_1-\cdots-\mu_p$ for some $-\gamma \in T^u \subseteq \Phi^-$ and $-\mu_i\in \Phi^-$, for all $i$ (see [\cite{humphreys}, Theorem $20.2$]).
 This implies  $\mu_1+\cdots+\mu_p-\alpha_l =-\gamma$. But this isn't possible since $-\alpha_l\neq -\gamma \in T^u$, 
 and since we can't sum positive roots $\mu_1,...,\mu_p$ and subtract a single simple root $\alpha_l$  to get a negative root different from $-\alpha_l$. Therefore, the weight of $e_{-\alpha_l} v_{\lambda_l}$ doesn't occur among weights of   $\mathfrak{s}_{T^u,0}\cdot V(\lambda_l)$.   In addition, it is a highest
weight vector with respect to  $\mathfrak{s}_{T^r,\mathfrak{k}}$.
 
Clearly $v_{\lambda_l}$, of weight $\lambda_l$,  is a highest weight vector with respect to $\mathfrak{s}_{T^r,\mathfrak{k}}$.  
And, all 
 weights of $\mathfrak{s}_{T^u,0}\cdot V(\lambda_l)$
 are strictly less than $\lambda_l$, hence the weight of $v_{\lambda_l}$ is not among those of $\mathfrak{s}_{T^u,0}\cdot V(\lambda_l)$.

To summarize what we've established so far:
\begin{itemize}
\item $e_{-\alpha_l} v_{\lambda_l} \notin  \mathfrak{s}_{T^u,0}\cdot V(\lambda_l)$ is a highest weight vector with respect to $\mathfrak{s}_{T^r,\mathfrak{k}}$; and
\item  $v_{\lambda_l} \notin \mathfrak{s}_{T^u,0}\cdot V(\lambda_l)$ is 
a highest weight vector with respect to $\mathfrak{s}_{T^r,\mathfrak{k}}$.
\item Moreover, $e_{-\alpha_l} v_{\lambda_l}$  and $v_{\lambda_l}$ are of different weights, and neither of these weights  are among the weights of  $\mathfrak{s}_{T^u,0}\cdot V(\lambda_l)$.
\end{itemize}
Let $[v_{\lambda_l}]$ and $[e_{-\alpha_l} v_{\lambda_l}]$ be the elements of $V(\lambda_l)/  \mathfrak{s}_{T^u,0}\cdot V(\lambda_l)$ with 
representative elements $v_{\lambda_l}$ and $e_{-\alpha_l} v_{\lambda_l}$, respectively.
Hence, we've established that $[v_{\lambda_l}]$ and $[e_{-\alpha_l} v_{\lambda_l}]$ are linearly independent, highest weight vectors, with respect to the semisimple regular subalgebra $\mathfrak{s}_{T^r,\mathfrak{k}}$ in
$V(\lambda_l)/  \mathfrak{s}_{T^u,0}\cdot V(\lambda_l)$.
Thus, $V(\lambda_l)/  \mathfrak{s}_{T^u,0}\cdot V(\lambda_l)$ has a proper decomposition  as 
an $\mathfrak{s}_{T^r,\mathfrak{k}}$-module. 

 This, however, 
 implies that $V(\lambda_l)/  \mathfrak{s}_{T^u,0}\cdot V(\lambda_l)$ isn't a simple $\mathfrak{s}_{T^r,\mathfrak{k}}$-module, a contradiction to the initial assumption that $\stt$ is cyclic wide.
 Therefore, it must be the case that $\cww=\Phi$.
\end{proof}

As we'll see in the next straightforward proposition, a perfect  regular Levi decomposable subalgebra 
has a radical that is an ad-nilpotent subalgebra. Hence, Theorem \ref{theoremtwo} applies to  perfect  regular Levi decomposable subalgebras.

\begin{proposition}\label{lemma2}
Let $\stt=\mathfrak{s}_{T^r,\mathfrak{k}} \inplus \mathfrak{s}_{T^u,\mathfrak{k}^\perp}$ be a perfect regular Levi decomposable subalgebra of $\mathfrak{g}$. Then,
\begin{equation}
\mathfrak{k}^\perp=\{0\}, ~T^u\neq \emptyset, ~\text{and}
\end{equation}
$\mathfrak{t}$  is the subalgebra of $\mathfrak{h}$ generated by $[\mathfrak{g}_\alpha, \mathfrak{g}_{-\alpha} ]$ for each $\alpha \in T^r$.
In particular, its radical is an ad-nilpotent subalgebra.
\end{proposition}
\begin{proof}
Since $\stt$ is  perfect (i.e., $\stt=[\stt, \stt]$), it follows immediately that $\mathfrak{k}^\perp=\{0\}$. This necessitates that $\mathfrak{t}$  is the subalgebra of $\mathfrak{h}$ generated by $[\mathfrak{g}_\alpha, \mathfrak{g}_{-\alpha} ]$ for each $\alpha \in T^r$. Further, 
 since $\stt$ is a regular Levi decomposable subalgebra, then 
 $T^u \neq \emptyset$ (see Remark \ref{remark}). 
\end{proof}

The following corollary establishes that each Levi decomposable subalgebra in \cite{dr15} is cyclic wide. As mentioned above, this result was
established in \cite{casati} for the special case of subalgebras of $\mathfrak{sl}_{n+1}$ isomorphic to $\mathfrak{sl}_n\inplus \mathbb{C}^n$, albeit by more technical means, relying on  
the Feigin-Fourier-Littelmann basis for $\mathfrak{sl}_{n+1}$-modules \cite{feigin}.

\begin{corollary}\label{widecyclicb}
Any subalgebra isomorphic to a subalgebra listed below is cyclic wide in the corresponding ambient simple Lie algebra:
\begin{equation}\label{ghwer}
\begin{array}{llllllllllll}
\mathfrak{sl}_{n} &\inplus& \mathbb{C}^n&\subset& \mathfrak{sl}_{n+1}, \\
\mathfrak{so}_{2n+1} &\inplus& \mathbb{C}^{2n+1} &\subset& \mathfrak{so}_{2n+3},~\text{and} \\
\mathfrak{so}_{2n} &\inplus& \mathbb{C}^{2n}  &\subset& \mathfrak{so}_{2n+2}.
\end{array}
\end{equation}
 (Note that  there are no subalgebras isomorphic to $\mathfrak{sp}_{2n}\inplus \mathbb{C}^{2n}$ in $\mathfrak{sp}_{2n+2}$ \cite{dr15}).
\end{corollary}
\begin{proof}
The subalgebras of $\mathfrak{sl}_{n+1}$, $\mathfrak{so}_{2n+3}$, and $\mathfrak{so}_{2n+2}$ isomorphic to $\mathfrak{sl}_{n} \inplus \mathbb{C}^{n}$, $\mathfrak{so}_{2n+1}\inplus \mathbb{C}^{2n+1}$, and $\mathfrak{so}_{2n}\inplus \mathbb{C}^{2n}$, respectively, were classified, up to inner automorphism in \cite{dr15}. All subalgebras  are regular Levi decomposable subalgebras and satisfy the condition 
 $\cww =\Phi$. Hence, by Theorem \ref{theorem1}, each subalgebra is cyclic wide. Note that the subalgebras are also perfect.
\end{proof}

We present two additional examples to further illustrate the utility of our results.

\begin{example}{\bf Some cyclic wide subalgebras of $\mathfrak{sl}_{n+1}$.}
The special linear algebra   $\mathfrak{sl}_{n+1}$ is a simple Lie algebra of rank $n$. It has simple roots $\Delta= \{ \alpha_1,..., \alpha_n\}$, and  positive roots 
\begin{equation}
\Phi^+= \{ \alpha_{p, q} \coloneqq \alpha_p+\cdots + \alpha_q~|~ 1 \leq p \leq q \leq n\}.
\end{equation}
Note that $\alpha_i =\alpha_{i,i}$. Define a closed subset $T_k$ of  $\Phi$, for all $k$ with $1\leq  k \leq n-1$, by specifying its symmetric and special parts, respectively:
\begin{equation}
\begin{array}{llll}
 T_k^r\coloneqq \{ \pm\alpha_{i,j}~|~   k+1 \leq i \leq j \leq n \}, \\
   T_k^u\coloneqq \{ \alpha_{i,j}~|~   1 \leq i \leq k, i\leq j \leq n \}.
   \end{array}
\end{equation}
Let $\mathfrak{t}_k$ be the subalgebra of $\mathfrak{h}$ generated by $[\mathfrak{g}_{\alpha}, \mathfrak{g}_{-\alpha}]$ 
for $\alpha\in T_k^r$, where $\mathfrak{g}= \mathfrak{sl}_{n+1}$.
Then, $\mathfrak{s}_{T_k, \mathfrak{t}_k}$ is a regular Levi decomposable subalgebra of $\mathfrak{sl}_{n+1}$.  Its Levi
factor $\mathfrak{s}_{T^r_k, \mathfrak{t}_k}$ is isomorphic to $\mathfrak{sl}_{n+1-k}$. Observe that the radical
$\mathfrak{s}_{T^u_1, 0}$ is abelian, but $\mathfrak{s}_{T^u_k, 0}$ is not abelian for $k>1$.

Further,  $T_k \cup -T_k =\Phi$. Hence, by Theorem \ref{theorem1},  $\mathfrak{s}_{T_k, \mathfrak{t}_k}$ is 
cyclic wide. Note that  $\mathfrak{s}_{T_1, \mathfrak{t}_1}$  is included in Corollary \ref{widecyclicb}.
\end{example}

\begin{example}{\bf Some cyclic wide subalgebras of non-simple, semisimple Lie algebras.} 
Let $\mathfrak{g}^i$ be a simple Lie algebra with corresponding root system $\Phi^i$ and Cartan subalgebra $\mathfrak{h}^i$, for each $i =1,2,...,m$.
Let  $T^i$ be a closed subset of $\Phi^i$, and $\mathfrak{t}^i$ 
the subalgebra of $\mathfrak{h}^i$ generated by $[\mathfrak{g}^i_{\alpha}, \mathfrak{g}^i_{-\alpha}]$ 
for each $\alpha\in (T^i)^r$. Suppose that $\mathfrak{s}_{T^i, \mathfrak{t}^i}$ is a regular Levi decomposable Lie algebra and $T^i \cup -T^i=\Phi^i$,  for each $i$. 
Then, 
\begin{equation}
\mathfrak{s}_{T^1,\mathfrak{t}^1} \oplus \cdots \oplus \mathfrak{s}_{T^m,\mathfrak{t}^m}
\end{equation} 
is a cyclic wide regular subalgebra of the semisimple Lie algebra $\mathfrak{g}^1 \oplus \cdots \oplus \mathfrak{g}^m$ by Theorem \ref{theorem1}.
\end{example}

\end{document}